\documentclass[11pt]{article}

\usepackage[margin=1in]{geometry}

\usepackage{lmodern}
\usepackage[T1]{fontenc}

\usepackage[english]{babel} 
\usepackage[utf8]{inputenc}

\usepackage{mathrsfs} 
\usepackage{amsfonts}
\usepackage{marvosym} 
\usepackage{amsmath,amssymb,amsthm}

\usepackage{bm}

\usepackage[usenames, dvipsnames, svgnames]{xcolor}
\definecolor{vert}{RGB}{15,120,5}
\definecolor{gris}{RGB}{128,128,128}
\definecolor{bleu}{RGB}{0,50,150}
\definecolor{rouge}{RGB}{149,24,24}
\usepackage[colorlinks = true,linkcolor=bleu,citecolor=vert]{hyperref}

\usepackage[nameinlink]{cleveref} 
\crefname{equation}{}{}
\usepackage{bbm}

\usepackage{epsfig}

\usepackage{tikz-cd} 

\usepackage{stmaryrd}
\usepackage{manfnt}
\usepackage{multicol}
\usepackage{array}
\usepackage{multirow}

\usepackage{fancyhdr} 
\usepackage{fancybox} 

\usepackage{calligra}
\title{Artin motives in relative Nori and Voevodsky motives.}
\author{Swann Tubach}
\date{}

\usepackage{tikz}
\usetikzlibrary{calc, intersections}

\usepackage{adjustbox}

\theoremstyle{plain}
\newtheorem{thm}{Theorem}[subsection]
\newtheorem{prop}[thm]{Proposition}

\newtheorem*{thm*}{Theorem}

\newtheorem*{prop*}{Proposition}
\newtheorem{nota}[thm]{Notation}

\newtheorem{lem}[thm]{Lemma}
\newtheorem{cor}[thm]{Corollary} 

\theoremstyle{definition}
\newtheorem{defi}[thm]{Definition}
\newtheorem*{defi*}{Definition}

\newtheorem{constr}[thm]{Construction}

\theoremstyle{remark}

\newtheorem{rem}[thm]{Remark}
\newtheorem*{rem*}{Remark}

\numberwithin{equation}{thm}

\newcommand{\C}{\mathbb{C}}

\newcommand{\Q}{\mathbb{Q}}
\newcommand{\Z}{\mathbb{Z}}
\newcommand{\N}{\mathbb{N}}

\newcommand{\Spec}{\mathrm{Spec}}
\newcommand{\A}{\mathbb{A}}

\newcommand{\HH}{\mathrm{H}}

\newcommand{\acal}{\mathcal{A}}

\newcommand{\ccal}{\mathcal{C}}

\newcommand{\mcal}{\mathcal{M}}

\newcommand{\Dd}{\mathrm{D}}

\newcommand{\colim}{\mathrm{colim}}

\newcommand{\Hom}{\mathrm{Hom}}

\newcommand{\map}[1]{\overset{#1}{\to}}

\newcommand{\Mc}{{\mathcal{M}_{\mathrm{ord}}}}
\newcommand{\Mp}{{\mathcal{M}_{\mathrm{perv}}}}

\DeclareMathOperator{\sHom}{\mathscr{H}\text{\kern -3pt {\calligra\large om}}\,}

\newcommand{\Sch}{\mathrm{Sch}}
\newcommand{\op}{\mathrm{op}}

\newcommand{\catinfty}{\mathrm{Cat}_\infty}

\usepackage{appendix}

\makeatletter
\newcommand{\subjclass}[2][2020]{%
  \let\@oldtitle\@title%
  \gdef\@title{\@oldtitle\footnotetext{#1 \emph{Mathematics subject classification.} #2}}%
}
\newcommand{\keywords}[1]{%
  \let\@@oldtitle\@title%
  \gdef\@title{\@@oldtitle\footnotetext{\emph{Key words and phrases.} #1.}}%
}
\makeatother

\begin{document}
\maketitle

\begin{abstract}
Over a scheme of finite type over a field of characteristic zero, we prove that Nori an Voevodsky categories of relative Artin motives, that is the full subcategories generated by the motives of étale morphisms in relative Nori and Voevodsky motives, are canonically equivalent. As an application, we show that over a normal base of characteristic zero an Artin motive is dualisable if and only if it lies in the thick category spanned by the motives of finite étale schemes. We finish with an application to motivic Galois groups and obtain an analogue of the classical exact sequence of étale fundamental groups relating a variety over a field and its base change to the algebraic closure. 
\end{abstract}
\tableofcontents
\subsection{Introduction.}
Let $k$ be a field of characteristic zero. As shown by Orgogozo in \cite{MR2102056} following Voevodsky \cite{MR1764197}, the category $\mcal^0(k)$ of \emph{Artin motives} over $k$, also known as Artin representations with rational coefficients of the absolute Galois group $\mathrm{Gal}(\overline{k}/k)$ of $k$ form a semi-simple abelian category which fits naturally as the heart of a \emph{motivic t-structure} on a full subcategory $\mathrm{DM}^0_c(k)$ of the triangulated category of Voevodsky motives $\mathrm{DM}_\mathrm{gm}(k)$ over a field. Moreover it is not hard to see that that $\mcal^0(k)$ embeds fully faithfully in the abelian category $\mcal(k)$ of \emph{Nori motives}, as shown by Huber and Müller-Stachs in \cite[Section 9.4]{MR3618276} or \cite[Section 4]{MR3302623}. By the work of Arapura (see \cite[Section 10.2.2]{MR3618276}) the abelian category $\mcal(k)$ admits a good theory of weights, and in particular its full subcategory of pure objects is semi-simple, and stable under extensions and subquotients in $\mcal(k)$. Because every Artin motive is of weight zero (indeed all of its objects are subquotients of $\HH^0(X,\Q)$ with $X$ finite étale over $k$), we see for all $i>0$ the group 
\[\Hom_{\Dd^b(\mcal(k))}(M,N[i])=0\] vanishes if $M$ and $N$ are Artin motives. In particular the canonical functor 
\[\Dd^b(\mcal^0(k))\to \Dd^b(\mcal(k))\] is an equivalence, and we can identify its image as the thick full subcategory generated\footnote{That is, the smallest $\infty$-category that contains the given generators and is stable under finite colimits, finite limits and direct factors.} by the $f_*\Q_X\simeq \HH^0(X,\Q)$ for $f\colon X\to\Spec(k)$ finite and étale. Here, we used the pushforward constructed by Ivorra and Morel in \cite{ivorraFourOperationsPerverse2022}. The purpose of this note is to tell a similar story for \emph{relative Nori motives}, as constructed by Ivorra and Morel.\\

Let $X$ be a finite type $k$-scheme\footnote{In fact one could take any finite dimensional qcqs scheme of characteristic zero, because the derived category of perverse Nori motives and all considered full subcategories afford continuity thus all statements that hold for varieties over $\Q$ holds in this greater generality. We choose not to pursue this here, but this will be developed in \cite{integralNori} where those are the schemes we consider naturally.}. F. Ivorra and S. Morel have constructed in \cite{ivorraFourOperationsPerverse2022} an abelian category of \emph{perverse Nori motives} $\Mp(X)$ on $X$ whose derived category $\Dd^b(\Mp(X))$ affords the six operations (the tensor product and the internal homomorphisms were constructed by L. Terenzi in \cite{terenziTensorStructurePerverse2024}). In our previous work \cite{SwannRealisation} we proved that this construction had a natural lift to the work of $\infty$-categories, and this is the language we will use in this note. In particular, if one denote by $\mathrm{DN}(X)$ the indization of the derived category of perverse Nori motives on $X$, we have a functor 
\[\Sch_k^\op\to \mathrm{CAlg}(\mathrm{Pr}^\mathrm{L})\] with values in presentably symmetric monoidal $\infty$-categories sending a finite type $k$-scheme to $\mathrm{DN}(X)$ and a morphism $f$ of scheme to the pullback functor $f^*$. Denote by $\mathrm{DM}(X)$ the presentable $\infty$-category of Voevodsky étale motives with rational coefficients (but without transfers). This category consists of the $\mathbb{P}^1$-stabilisation of the $\infty$-category of $\A^1$-invariants étale hypersheaves on the category $\mathrm{Sm}_X$ of smooth $X$-schemes (see \cite{MR3205601}, \cite{MR3477640} and \cite{robaloThese}). By \cite{SwannRealisation} there is a functor 
\[\rho_\mathrm{N}\colon \mathrm{DM}\to \mathrm{DN} \] that commutes with the 6 operations. Both the source and target of this functor afford the continuity property and h-descent.


\begin{defi}
  \label{alldef}
Let $k$ be a field and let $X$ be a finite type $k$-scheme. For this definition, $\mathrm{D}(X)$ denotes either $\mathrm{DM}(X)$ or $\mathrm{DN}(X)$.
\begin{enumerate}
  \item The $\infty$-category $\Dd^0(X)$ of Artin motives in $\mathrm{D}(X)$ is the localising\footnote{That is, the smallest full subcategory that contains the generators and is stable under all colimits.} subcategory of $\mathrm{D}(X)$ spanned by the objects $f_*\Q_Y$ where $f:Y\to X$ is finite. We denote by $\Dd^0_c(X)$ its full subcategory of compact objects, which is the thick full subcategory spanned by the $f_*\Q_Y$ for $f:Y\to X$ finite.
  \item The $\infty$-category $\Dd^{0,\mathrm{sm}}(X)$ of smooth Artin motives in $\mathrm{D}(X)$ is the localising subcategory of $\mathrm{D}(X)$ spanned by the objects $f_*\Q_Y$ where $f:Y\to X$ is finite and étale. We denote by $\Dd^{0,\mathrm{sm}}_c(X)$ its full subcategory of compact objects, which is the thick full subcategory spanned by the $f_*\Q_Y$ for $f:Y\to X$ finite and étale.
  \item The $\infty$-category $\Dd^\mathrm{coh}(X)$ of cohomological motives in $\mathrm{D}(X)$ is the localising subcategory of $\mathrm{D}(X)$ spanned by the objects $f_*\Q_Y$ where $f:Y\to X$ is proper. We denote by $\Dd^\mathrm{coh}_c(X)$ its full subcategory of compact objects, which is the thick full subcategory spanned by the $f_*\Q_Y$ for $f:Y\to X$ proper.
  \item The $\infty$-category $\Dd^{0,\mathrm{rig}}(X)$ of rigid Artin motives in $\mathrm{D}(X)$ is full subcategory of $\Dd^0(X)$ spanned by the dualisable objects (seen as objects of $\Dd(X)$). Because in $\Dd(X)$ the unit object $\Q_X$ is compact, we have $\Dd^{0,\mathrm{rig}}(X)\subset \Dd^0_c(X)$.
\end{enumerate}

\end{defi}
We have the following inclusions:
\[\Dd^{0,\mathrm{sm}}_c(X)\subset\Dd^{0,\mathrm{rig}}(X)\subset \Dd^{0}_c(X)\subset \Dd^\mathrm{coh}_c(X)\subset \Dd_c(X).\] 
Indeed the only nontrivial one is the first, which holds because for $f:Y\to X$ finite and étale, the object $f_*\Q_Y$ is dualisable with dual itself (this follows from the projection formula). If $X=\Spec(k)$ is the spectrum of a field of characteristic zero, then the reverse inclusion $\Dd_c^0(X)\subset \Dd^{0,\mathrm{sm}}_c(X)$ holds.

By the work of S. Pépin-Lehalleur \cite{MR3920833} and \cite{MR4033829} the $\infty$-category $\mathrm{DM}^0(X)$ affords an ordinary motivic t-structure which restricts to compact objects $\mathrm{DM}^0_c(X)$, and for which each $\ell$-adic realisation \[\mathrm{DM}^0_c(X)\to \Dd^b_c(X_{\'et},\Q_\ell)\] is  t-exact (and conservative) when $\Dd^b_c(X_{\'et},\Q_\ell)$ is endowed with the ordinary t-structure whose heart is the abelian category of constructible sheaves.

Because the functor $\rho_\mathrm{N}$ commutes with the 6 operations, it sends Artin motives to Artin motives, cohomological motives to cohomological motives, rigid objects to rigid objects \emph{etc.} Moreover, the conservativity and t-exactness of the $\ell$-adic realisation on Nori motives implies that the functor
\[\rho_\mathrm{N}\colon\mathrm{DM}^0(X)\to \mathrm{DN}(X)\] is t-exact. Our first result is the comparison of the big categories of Artin motives:
\begin{thm}[\Cref{thmA}]
    Let $k$ be a field of characteristic zero and let $X$ be a finite type $k$-scheme. The functor 
    $$\mathrm{DM}^0(X)\to\mathrm{DN}^0(X)$$
    is an equivalence.
\end{thm}
The proof consists in first proving the result for lisse motives, and then use the localisation sequence combined with Nair's construction of the Artin truncation functor $\omega^0$ to do an Noetherian induction.
Our second result uses an adaptation of a proof of Haas (\cite{HaasThesis}) in the case of $1$-motives to Artin motives and is a generalisation to normal schemes (but rational coefficients) of the result of Ruimy \cite[Theorem 2.2.2.1]{ruimyAbelianCategoriesArtin2023} over regular schemes:
\begin{thm}
  Let $X$ be a normal and connected finite type $k$-scheme. The inclusion 
  \[\mathrm{DM}^{0,\mathrm{sm}}_c(X)\to\mathrm{DM}^{0,\mathrm{rig}}(X)\] is an equivalence, the t-structure on $\mathrm{DM}^0_c(X)$ restricts to $\mathrm{DM}^{0,\mathrm{sm}}_c(X)$ and the heart is canonically identified (through the Artin motive functor) with $\mathrm{Rep}_\Q^A(\pi_1^{\'et}(X))$ the abelian category of Artin representations of the étale fundamental group of $X$.
\end{thm}
We finish with some application to motivic Galois groups in \Cref{exactDiagMotGroup}.

This work will be used in \cite{integralNori} to prove the same statements with integral coefficients, and to extend Ruimy's main theorem \cite[Theorem 3.2.3.7]{ruimyAbelianCategoriesArtin2023} to the characteristic zero case: the Artin vanishing theorem for the perverse homotopy structure on Artin motives.
\subsection{Artin motives over a field.}

Let $k$ be a field of characteristic zero. As reminded in the introduction we have:
\begin{thm}[Ayoub -- Barbieri-Viale -- Huber -- Müller-Stachs -- Orgogozo -- Pepin-Lehalleur]
	Let $k$ be a field a characteristic zero. The functor \[\rho_\mathrm{N}^0:\Dd^b(\mathrm{Rep}^A_\Q(\mathrm{Gal}(\overline{k}/k)))\simeq \mathrm{DM}_c^0(k)\to\mathrm{DN}_c^0(k)\] 
	is an equivalence of stable $\infty$-categories. Moreover, the composition with the inclusion 
	$$\Dd^b(\mathrm{Rep}^A_\Q(\mathrm{Gal}(\overline{k}/k)))\to\Dd^b(\mcal(k))$$
	 is t-exact. 
\end{thm} 
In the theorem we have denoted by $\mathrm{Rep}^A_\Q(\mathrm{Gal}(\overline{k}/k))$ the abelian category of Artin representations of the absolute Galois group of $k$, that is finite dimensional representation $(V,\rho)$ such that the induced morphism $\mathrm{Gal}(\overline{k}/k)\to\mathrm{GL}(V)$ factors through a finite quotient, a category that we also have denoted by $\mcal^0(k)$. We also have denoted by $\mcal(k)$ the abelian category of cohomological Nori motives over $k$, which coincides with the category $\Mp(\Spec(k))$ of perverse Nori motives over $\Spec(k)$.

We begin with a statement about cohomological Nori motives.
\begin{prop}
  \label{tstrctufields}
The t-structure on $\Dd^b(\mcal(k))$ restricts to $\mathrm{DN}^\mathrm{coh}_c(k)$. We denote by $\mcal^\mathrm{coh}(k)$ its heart.
\end{prop}
\begin{proof}
Let $\acal$ be the full subcategory of $\mcal(k)$ whose objects are the $M\in\mcal(k)$ satisfying the following property $(*)$: For all $w\in\Z$ the graded piece $\mathrm{Gr}^W_w M$ is a direct sum of subquotients of objects of the form $\HH^i(f_*\Q_X)$ for $i\in\Z$ and $f\colon X\to\Spec(k)$ smooth and projective. Because for such an $f$ the object $\HH^i(f_*\Q_X)$ is pure of weigh $i$, it turns out that only $i=w$ appears in the property. Moreover, because the weight truncation functors are exact and the categories of objects pure of a certain weight $w\in\Z$ are semi-simple, the category $\acal$ is a Serre abelian subcategory of $\mcal(k)$, that is it is table under extensions and subquotients. In particular, the t-structure on $\mathrm{DN}_c(k)$ restricts to the stable $\infty$-category $\ccal$ consisting of complexes having cohomology sheaves in $\acal$. By the decomposition theorem we have $\mathrm{DN}^\mathrm{coh}_c(k)\subset \ccal$, and by dévissage and the decomposition theorem again, the reverse inclusion is also true, finishing the proof.
\end{proof}
\begin{rem}
The above proposition works verbatim if one replaces cohomological motives $\mathrm{DN}^\mathrm{coh}_c(k)$ with $n$-motives $\mathrm{DN}^n_c(k)$, the thick category spanned by the $f_*\Q_X$ for $f$ a smooth proper morphism of relative dimension $\leqslant n$, providing abelian categories of $n$-motives $\mcal^n(k)$ for all $n\in\N$.
\end{rem}

We use the ideas of Morel \cite{MR2350050}, Nair and Vaish \cite{MR3293216}, and Ruimy \cite{ruimyArtinPerverseSheaves2023} to have a well behaved Artin truncation functor on $\mathrm{DN}^\mathrm{coh}_c(k)$. Note that at first, there are \emph{a priori} two ways of producing an Artin motive out of a cohomological motive. The first one is tautological: the inclusion $\mathrm{DN}^0(k)\to\mathrm{DN}^\mathrm{coh}(k)$ preserves colimits by definition, hence admits a right adjoint $\omega^0$ by Lurie's adjoint functor theorem. The second one is inspired by \cite{MR2350050}:
\begin{constr}
  Denote by \[\mathrm{DN}_{c}(k)^{ \leqslant \mathrm{Id}} := \{M\in\mathrm{DN}_{c}(k)\mid \forall i\in\Z, \HH^i(M) \text{ is of weights }\leqslant 0\} \] 
  and by and by 
  \[\mathrm{DN}_{c}(k)^{> \mathrm{Id}} := \{M\in\mathrm{DN}_{c}(k)\mid \forall i\in\Z, \HH^i(M) \text{ is of weights }> 0\} .\]
  By \cite[Proposition 3.1.1]{MR2350050}, this form a t-structure $\mathrm{t}^\mathrm{Id}$ on $\mathrm{DN}_{c}(k)$. Denote
   by $\tau_c^{\leqslant \mathrm{Id}}$ be the truncation functor right adjoint to 
  the inclusion $\mathrm{DN}_{c}(k)^{\omega \leqslant 0}\to\mathrm{DN}_c(k)$. 
  As the t-structure on $\mathrm{DN}(k)$ restricts to $\mathrm{DN}_c(k)$, we see
   that the t-structure $\mathrm{t}^\mathrm{Id}$
   of Morel on $\mathrm{DN}_c(k)$ induces one on $\mathrm{DN}(k)$ and the associated functor \[\tau^{\leqslant\mathrm{Id}}\colon \mathrm{DN}(k)\to\mathrm{DN}(k)^{\leqslant \mathrm{Id}}\] preserves compact objects, giving back the above $\tau_c^{\leqslant\mathrm{Id}}$. Moreover, $\tau^{\leqslant\mathrm{Id}}$ is an exact functor of stable $\infty$-categories.
  \end{constr}

  \begin{prop}
    \label{weightTruncationfield}
    We and an equality 
    \[\mathrm{DN}^0(k) = \mathrm{DN}^\mathrm{coh}(k)\cap\mathrm{DN}(k)^{\leqslant \mathrm{Id}}\] of full subcategories of $\mathrm{DN}(k)$. Moreover, the restriction of $\tau^{\leqslant\mathrm{Id}}$ to $\mathrm{DN}^\mathrm{coh}(k)$ coincides with $\omega^0$, so that the latter preserves compact objects and is t-exact.
  \end{prop}
  \begin{proof}
    Because the generators of $\mathrm{DN}^0(k)$ are in $\mathrm{DN}(k)^{\leqslant\mathrm{Id}}$, we have a trivial inclusion 
    \[\mathrm{DN}^0(k) \subset\mathrm{DN}^\mathrm{coh}(k)\cap\mathrm{DN}(k)^{\leqslant \mathrm{Id}}.\] To prove the reverse inclusion, it suffices to prove that if $M=f_*\Q_X$ is a generator of $\mathrm{DN}^\mathrm{coh}(k)$, meaning that $f\colon X\to\Spec(k)$ is smooth and proper, then the object $\tau^{\leqslant\mathrm{Id}}f_*\Q_X$ is an Artin motive.
    Denote by $X\map{\alpha}\pi_0(X)\map{\pi}\Spec k$ the Stein factorisation 
		  of $f$ (see \cite[Corollaire 4.3.3]{MR0217085}). There is a natural map $\pi_*\Q_{\pi_0(X)}\to f_*\Q_X$, which is an isomorphism on $\HH^0$, as it can be seen after realisation ($\HH^0(\pi_0(X))\simeq \HH^0(X)$ is of rank the number of connected components of $X$.)
		  It is also an isomorphism on $\HH^i$ for $i<0$ because pushforward are left exact. 
		As $\HH^i(\pi_*\Q_{\pi_0(X)}) = 0$ if $i>0$ because $\pi$ is finite étale hence $\pi_*$ is t-exact, we have in fact that the cofiber $C$
		of the map $\pi_*\Q_{\pi_0(X)}\to f_*\Q_X$ verifies $\HH^i(C)=0$ if $i\leqslant 0$ and $\HH^i(C)\simeq\HH^i(f_*\Q_X)$ for $i>0$. The cofiber 
		sequence \[\pi_*\Q_{\pi_0(X)}\to f_*\Q_X \to C\] is therefore equivalent to the truncation cofiber sequence 
		\[\tau^{\leqslant 0}f_*\Q_X\to f_*\Q_X \to \tau^{>0}f_*\Q_X \]
		for the canonical t-structure on $\Dd^b(\mcal(k))$. For $i\geqslant 1$, the weight of $\HH^i(f_*\Q_X)$ is $i$ and the weight of $\pi_*\Q_{\pi_0(X)}$ is zero, hence 
		this cofiber sequence is also the truncation cofiber sequence for the Morel t-structure: the map $\pi_*\Q_{\pi_0(X)}\to f_*\Q_X$ coincides with 
		the map $\tau^{\leqslant \mathrm{Id}}f_*\Q_X\to f_*\Q_X$. In particular, $\tau^{\leqslant \mathrm{Id}}M$ is an Artin motive. The fact that \[\omega^0\simeq (\tau^{\leqslant\mathrm{Id}})_{\mid \mathrm{DN}^\mathrm{coh}}\] can be seen by a direct application of the Yoneda lemma. Finally, the t-exactness of $\omega^0$ follows from the t-exactness of the weight truncation functors.
  \end{proof}
  \begin{cor}
    \label{compatiOmega0}	Let $\omega^0:\mathrm{DM}^\mathrm{coh}(k)\to\mathrm{DM}^0(k)$ be the right adjoint to the inclusion constructed by \cite{MR3920833}, \cite{MR4109490} and \cite{MR2494373}. 
    Then the natural map $\rho_\mathrm{N}\circ\omega^0\to\omega^0\circ \rho_\mathrm{N}$ is an equivalence..
  \end{cor}
  \begin{proof}
    As $\rho_\mathrm{N}$ sends Artin motives to Artin motives we have a natural transformation $\rho_\mathrm{N}\circ\omega^0\to \omega^0\circ \rho_\mathrm{N}$. To show that this 
    map is an equivalence it suffices to check that this natural transformation is an equivalence when evaluated at the generators $f_*\Q_X$ of cohomological motives, where $f:X\to \Spec k$ is proper and smooth.
    This is the case by commutation of $\rho_\mathrm{N}$ with the operations and because of the computation $\omega^0f_*\Q_X\simeq \pi_*\Q_{\pi_0(X)}$ with $\pi:\pi_0(X)\to\Spec(k)$ the Stein factorisation of $X$. 
    Indeed for Nori motives this has been proven in the previous proposition, and for Voevodsky motives this is \cite[Proposition 3.7]{MR3920833}.
  \end{proof}
  \subsection{General permanence properties.}
  In order to do efficient dévissage, we will need our categories to have a good functorial behaviour. We start with stability by the operations:
  \begin{prop}
    \label{stabop}
    Let $k$ be a field of characteristic zero and let $f\colon Y\to X$ be a map of finite type $k$-schemes. Then
    \begin{enumerate}
    \item The functor $f^*$ preserves all categories defined in \Cref{alldef}. 
    \item The functors $f_*$ and $f_!$ preserve $\mathrm{DN}^\mathrm{coh}$ and $\mathrm{DM}^\mathrm{coh}$.
    \item If $f$ is quasi-finite the functor $f_!$  preserves $\mathrm{DN}^0$ and $\mathrm{DM}^0$ 
	\item The functor $f^!$ preserves $\mathrm{DN}^\mathrm{coh}$ and $\mathrm{DM}^\mathrm{coh}$.
    \end{enumerate}  
  \end{prop}
  \begin{proof}
    In each case, it suffices to check that the functor sends generators to the correct category. In the case of $\mathrm{DM}$ this is \cite[Propositions 1.12 and 1.14]{MR3920833}, if one notes that we work in characteristic zero so that we even have resolutions by singularities. By commutation of $\rho_\mathrm{N}$ with the operations, the result for $\mathrm{DM}$ implies the result for $\mathrm{DN}$. Pépin-Lehalleur results are for separated morphisms but we can do a Zariski covering to reduce to this case.
  \end{proof}
  To make the statement of the next proposition clearer we have to 
  use that $\mathrm{DN}$ can be extended to all finite dimensional qcqs $\Q$-schemes as a functor that sends limit of schemes with affine transitions to colimits of presentable $\infty$-categories. This follows from the presentations of $\mathrm{DN}$ as modules in $\mathrm{DM}$ over an algebra $\mathscr{N}\in\mathrm{CAlg}(\mathrm{DM}(\Q))$, proven in \cite{SwannRealisation}, and continuity of $\mathrm{DM}$ as proven in \cite[Lemma 5.1]{MR4319065}. In this setting, all the definitions of \Cref{alldef} make sense.
  \begin{prop}
    Let $X=\lim_i X_i$ be a limit of finite dimension qcqs $\Q$-schemes with affine transitions. Let $\ccal$ be one of the categories defined in \Cref{alldef}, and denote by $\ccal_c$ its compact version.
    Then the natural functor 
    \[\colim_i\ccal(X_i)\to \ccal(X)\] 
    (\emph{resp.} \[\colim_i\ccal_c(X_i)\to\ccal_c(X_i))\]
    is an equivalence in $\mathrm{Pr}^\mathrm{L}$ (\emph{resp.} in $\catinfty$.)
  \end{prop}
  \begin{proof}
    We first deal with the case of rigid objects. In that case the colimit is in $\catinfty$. Fully faithfulness follows from the fully faithfulness for the ambient category $\mathrm{DM}_c$ and $\mathrm{DN}_c$. Essential surjectivity follows from the fact that the dualisability data is a finite data, so that each part of the data (the dual, the evaluation, the co-evaluation \emph{etc.}) can be reach at a finite level $X_i$ by the essential surjectivity on the ambient category.

    For the other categories, it suffices to deal with the compact versions. Fully faithfulness is as above, a consequence of fully faithfulness of the ambient category. Essential surjectivity follows from the approximation properties of finite type properties, as in \cite[Lemma 1.24]{MR3920833}.
  \end{proof}
  \begin{cor}
    \label{lemmaSm}
    Let $k$ be a field of characteristic zero and let $X$ be a finite type $k$-scheme. Let $M\in\mathrm{DN}^0_c(X)$ (\emph{resp.} $M\in\mathrm{DM}^0_c(X)$). Then there exits a dense open subset $U\subset X$ such that the restriction $M_{\mid U}$ is in $\mathrm{DN}^{0,\mathrm{sm}}_c(U)$ (\emph{resp.} in $\mathrm{DM}^{0,\mathrm{sm}}_c(U)$).
  \end{cor}
  \begin{proof}
    Indeed because over a field we have that all $0$-motives are smooth, the continuity of smooth motives around the scheme of generic points of $X$ provides the wanted $U$.
  \end{proof}
  The same proofs as \cite[Proposition 1.19 and Proposition 1.25]{MR3920833} give :
\begin{prop}
	\label{stabnmot}
  Let $k$ be a field of characteristic zero and let $X$ be a finite type $k$-scheme and let $M\in\mathrm{DN}_c(X)$. The following are equivalent :
	\begin{enumerate}
		\item $M\in \mathrm{DN}_c^\mathrm{coh}(X)$ (\emph{resp.} in $\mathrm{DN}_c^{0}(X)$)
		\item There exists a closed subset $i:Z\to X$ with open complement $j:U\to X$ such that $j^*M\in \mathrm{DN}_c^\mathrm{coh}(U)$ (\emph{resp.} in $\mathrm{DN}_c^{0}(U)$)
		      and $i^*M \in \mathrm{DN}_c^\mathrm{coh}(Z)$ (\emph{resp.} in $\mathrm{DN}_c^{0}(Z)$).
		\item For all $x\in X$, $x^*M \in \mathrm{DN}_c^\mathrm{coh}(\kappa(x))$ (\emph{resp.} in $\mathrm{DN}_c^{0}(\kappa(x))$)
	\end{enumerate}
  The same result holds with $\mathrm{DN}$ replaced by $\mathrm{DM}$.
\end{prop}

Recall that the ordinary t-structure on $X$ is the unique t-structure on $\mathrm{DN}_c(X)$ such that all pullback functors are t-exact. It can be constructed from the perverse t-structure using gluing (see \cite[Remarks 4.6]{MR1047415}). We denote by $\Mc(X)$ its heart.
\begin{cor}
  Let $k$ be a field of characteristic zero and let $X$ be a finite type $k$-scheme. Then the ordinary t-structure on $\mathrm{DN}_c(X)$ restricts to  $\mathrm{DN}^0_c(X)$, $\mathrm{DN}^{0,\mathrm{rig}}(X)$ and $\mathrm{DN}^\mathrm{coh}_c(X)$. 
\end{cor}
\begin{proof}
  It restricts to $\mathrm{DN}^\mathrm{rig}(X)$ because the tensor product is t-exact in both variables. We prove that the t-structure restricts to $\mathrm{DN}^\mathrm{coh}_c(X)$, the case of Artin motives being similar. Let $M\in\mathrm{DN}^\mathrm{coh}_c(X)$. We want to show that $\HH^0(M)$ is still a cohomological motive. By the above \Cref{stabnmot} it suffices to prove that for all $x\in X$ the object $x^*\HH^0(M)$ is cohomological. This follows from t-exactness of $x$ and the case of fields \Cref{tstrctufields}.
\end{proof}

\subsection{The equivalence of Artin motives.}
In this section we prove that the functor $\rho_\mathrm{N}$ induces an equivalence on $0$-motives. 
We begin by the following computation:
\begin{prop}
  \label{etcoh}
  Let $k$ be a field of characteristic zero and let $X$ be a smooth $k$-scheme. Let $f\colon Y\to X$ be a finite étale map and $n\in\Z$ be an integer.
  Then the maps 
  \[\HH^n(Y_{\'et},\Q)\to \Hom_{\mathrm{DM}(X)}(\Q_X,f_*\Q_Y[n])\to \Hom_{\mathrm{DN}(X)}(\Q_X,f_*\Q_Y[n])\]
  induced by the functors 
  \[\Dd(X_{\'et},\Q)\to\mathrm{DM}(X)\to\mathrm{DN}(X)\] 
  are both equivalences.
\end{prop}
\begin{proof}
  We can assume $X$ connected.
  The computation for Voevodsky motives follows from \cite[Proposition 1.7.2 and Corollary 2.1.4]{ruimyAbelianCategoriesArtin2023}. For Nori motives, the cases $m\neq 0$ are easy: both the étale cohomology and the $\Hom$ group are vanishing, because of t-structure ($m<0$) and weight ($m>0$) reasons. For $m=0$ by adjunction we reduced to the case $f=\mathrm{Id}$ and it follows from the fact that the group of endomorphisms of the unit object is $\Q$ (this is clear if $k$ is a subfield of $\C$ because it is a nonzero $\Q$-vector space that embeds in $\Q=\mathrm{End}_{\mathrm{Shv}(X(\C))}(\Q_X)$, and then follows by continuity for general fields).
\end{proof}
\begin{cor}
  \label{0motsmooo}
  Let $k$ be a field of characteristic zero and let $X$ be a smooth $k$-scheme.
  The functor 
  \[\mathrm{DM}^{0,\mathrm{sm}}(X)\to \mathrm{DN}^{0,\mathrm{sm}}(X)\] induced by $\rho_\mathrm{N}$ is an equivalence.
\end{cor}
\begin{proof}
  It suffices to show that it is fully faithful on the compact objects, and even, as compact objects are the idempotent completion of the stable category generated by the $f_*\Q_Y$ for $f\colon Y\to X$ finite and étale, it suffices to prove that it is fully faithful on the smallest stable subcategory of $\mathrm{DM}(X)$ that contains the $f_*\Q_Y$ for $f\colon Y\to X$. By dévissage, it suffices to prove that for all $m\in\Z$, $f\colon Y\to X$ and $g\colon Z\to X$ both finite étale, the map 
  \[\Hom_{\mathrm{DM}(X)}(g_*\Q_Z,f_*\Q_Y[n])\to \Hom_{\mathrm{DN}(X)}(g_*\Q_Z,f_*\Q_Y[n])\] is an equivalence. Because $g_*=g_!$ and $g^!=g^*$, by adjunction and proper base change we can assume that $g=\mathrm{Id}_X$ and replace $Y$ by $Y\times_X Z$. In this case, the result is the above computation done in \Cref{etcoh}.
\end{proof}

We now deal with non necessarily smooth $0$-motives. For this, we will need the truncation functor of Vaish (see also \cite{MR4033829}) over a general basis.
\begin{prop}
	\label{weightTruncation}
  Let $k$ be a field of characteristic zero and let $X$ be a finite type $k$-scheme.
	The inclusion $\mathrm{DN}^0_c(X)\subset \mathrm{DN}^{\mathrm{coh}}_c(X)$ admits a right adjoint $\omega^0$. Moreover, $\rho_\mathrm{N}$ commutes with $\omega^0$.
\end{prop}
\begin{proof}
	The existence of $\omega^0$ for Nori motives is proven exactly in the same way as for étale motives in \cite[Theorem 5.2.2]{MR4109490}. This uses the existence 
	of a truncation functor over a field, a fact ensured by \Cref{weightTruncationfield}.  By definition,
	$\omega^0$ is a truncation functor for a $m$-structure (\cite[section 2.5]{MR4109490}), obtained by punctual gluing. As $\rho_\mathrm{N}^K$ is $m$-exact for all fields $K$ by \Cref{compatiOmega0},
	we have that $\rho_\mathrm{N}^X$ is also $m$-exact, hence commutes with truncations.
\end{proof}
We can now prove the main result of the section.
\begin{thm}
  \label{thmA}
  Let $k$ be a field of characteristic zero and let $X$ be a finite type $k$-scheme. The functor 
  \[\mathrm{DM}^0(X)\to\mathrm{DN}^0(X)\] is an equivalence.
\end{thm}
\begin{proof}
  As the functor clearly reaches the generators and preserves compact objects, it suffices to prove that it is an fully faithful on $\mathrm{DM}^0_c(X)$. We do a Noetherian induction on $X$. Let $M,N\in\mathrm{DM}^0_c(X)$.  By \Cref{lemmaSm} and generic smoothness, there exists a nonempty smooth open subset $j\colon U\to X$ such that both $j^*M$ and $j^*N$ are smooth. Denote by $i\colon Z\to X$ the inclusion of the reduced closed complement. By localisation, we have the following commutative diagram
  \[\begin{tikzcd}
    {\mathrm{map}_{\mathrm{DM}(X)}(i_*i^*N,M)} & {\mathrm{map}_{\mathrm{DM}(X)}(N,M)} & {\mathrm{map}_{\mathrm{DM}(X)}(j_!j^*N,M)} \\
    {\mathrm{map}_{\mathrm{DN}(X)}(i_*i^*N,M)} & {\mathrm{map}_{\mathrm{DN}(X)}(N,M)} & {\mathrm{map}_{\mathrm{DN}(X)}(j_!j^*N,M)}
    \arrow[from=1-1, to=1-2]
    \arrow[from=1-1, to=2-1]
    \arrow[from=1-2, to=1-3]
    \arrow[from=1-2, to=2-2]
    \arrow[from=1-3, to=2-3]
    \arrow[from=2-1, to=2-2]
    \arrow[from=2-2, to=2-3]
  \end{tikzcd}\] where the lines are cofiber sequences of spectra. By the $(j_!,j^*)$ adjunction and the smooth case \Cref{0motsmooo}, the right vertical map is an equivalence. Using the $(i_*,i^!)$ adjunction, the fact that $i$ being quasi-finite the functor $i^!$ preserves cohomological motives by \Cref{stabop} and the fact that the truncation functor $\omega^0$ commutes with $\rho_\mathrm{N}$ by \Cref{weightTruncation}, we see that the left vertical map is equivalent to the map 
  \[\mathrm{map}_{\mathrm{DM}(Z)}(i^*N,\omega^0i^!M)\to \mathrm{map}_{\mathrm{DN}(Z)}(i^*N,\omega^0i^!M)\] which is an equivalence by Noetherian induction, finishing the proof.
\end{proof}
\subsection{Smooth and rigid Artin motives.}
In this section we prove that smooth and rigid Artin coincide on a normal scheme, and we identify the heart of the t-structure. We fix a field $k$ of characteristic zero.

\begin{lem}
	\label{omega0support}
Let $i:Z\to S$ be a closed immersion of finite type $k$-schemes with $S$ regular, such that the codimension of $Z$ in $S$ is $\geqslant 1$.
 Then for any $M\in\mathrm{DN}^{0,\mathrm{rig}}_c(S)$ we have $\omega^0(i^!M)=0$.
\end{lem}	
\begin{proof}
	We argue as in \cite[Lemma 2.4]{MR4033829}: If $i$ was a regular immersion of codimension $c>0$, then by absolute purity as in \cite[Proposition 1.7]{MR3920833} we would 
	have an isomorphism $i^!M\simeq i^*M(-c)[-2c]$. Now, by \Cref{thmA}, $M = \rho_\mathrm{N}(M')$ for some $M'\in \mathrm{DM}^{0}_c(S)$. 
	By \cite[Corollary 3.9 (iii)]{MR3920833}, $\omega^0(i^*M'(-c)[2c])=0$ (note that $i^!M'\in\mathrm{DM}^\mathrm{coh}$ by \cite[Proposition 1.12 (iv)]{MR3920833}) hence 
	$\omega^0(i^!M)=\rho_\mathrm{N}(0)=0$.
	Now if $i$ is not a regular immersion, we can find a dense open subset $u:U\to Z$, such that $U$ is regular with reduced complement $i_1:Z_1\to Z$. As $U$ and $S$ are regular,
	 \cite[\href{https://stacks.math.columbia.edu/tag/0E9J}{Tag 0E9J}]{stacks-project} ensures that $i\circ u$ is a regular 
	immersion of codimension $c\geqslant 1$. We have a cofiber sequence 
	\begin{eqnarray}
		(i_1)_*(i\circ i_1)^!M \to i^!M\to u_*(i\circ u)^!M
	\end{eqnarray}
	which gives $\omega^0((i_1)_*(i\circ i_1)^!M)\simeq \omega^0(i^!M)$, because $(i\circ u)^!M\simeq (i\circ u)^*M[-2c](-c)$ thus $\omega^0(u_*(i\circ u)^!M) = 0$ by the same argument as above using \cite[Corollary 3.9 (iii)]{MR3920833}. 
	As $i_1$ is quasi-finite and $(i_1)_*\simeq (i_1)_!$, the functors $(i_1)_*$ and $\omega^0$ commute, giving finally that $(i_1)_*\omega^0((i\circ i_1)^!M)\simeq \omega^0(i^!M)$.
	A Noetherian induction then implies that $\omega^0(i^!M)=0$.
\end{proof}
\begin{lem}
  \label{GoodRed}
  Let $S$ be a normal and connected finite type $k$-scheme with generic point $\eta$.
  Let $M\in\mathrm{Rep}_\Q^A(\pi_1^{\acute{e}t}(\eta))$ be an Artin representation such that there exists $N\in\mathrm{Rep}_{\Q_\ell}^A(\pi_1^{\acute{e}t}(S))$ verifying $N_\eta \simeq M\otimes_\Q\Q_\ell$. Then there exists an object $M'\in \mathrm{Rep}_\Q^A(\pi_1^{\acute{e}t}(S))$ such that $M'_\eta \simeq M$.
\end{lem}
\begin{proof}
  The object $M$ corresponds to a morphism $\pi_1^{\acute{e}t}(\eta)\to\mathrm{GL}_n(\Q)$. The hypothesis may be translated as the existence of a commutative square \[\begin{tikzcd}
    \pi_1^{\acute{e}t}(\eta) \arrow[r] \arrow[d] 
        & \mathrm{GL}_n(\Q) \arrow[d] \\
        \pi_1^{\acute{e}t}(S) \arrow[r]
        & \mathrm{GL}_n(\Q_\ell)
  \end{tikzcd}.\]
Now the right vertical map is injective, and the left vertical map is surjective (because $S$ is normal, by \cite[Exposé V. Proposition 8.2]{SGA1}). Thus in fact the image of the lower horizontal map lands in $\mathrm{GL}_n(\Q)$, providing the seeked $M'$.
\end{proof}
Recall that by Cavicchi, Déglise and Nagel \cite[Proposition 2.3.4]{MR4631432} the canonical functor 
\[\Dd^b(\mathrm{Rep}^A_\Q(\pi_1^{\'et}(S))\to\mathrm{DM}^{0,\mathrm{sm}}_c(S)\] is an equivalence when $S$ is a regular connected scheme.  In particular the stable $\infty$-category $\mathrm{DM}^{0,\mathrm{sm}}_c(S)$ have a t-structure.
\begin{prop}[Haas]
  \label{pi1etale}
	Let $S$ be a connected smooth $k$-scheme of finite type. Let $M\in\mathrm{DM}^0_c(S)^\heartsuit$. The following are equivalent :
	\begin{enumerate}
		\item $M\in\mathrm{DM}^{0,\mathrm{sm}}_c(S)^\heartsuit \simeq \mathrm{Rep}^A_\Q(\pi^1_{\acute{e}t}(S))$.
		\item $M$ is dualisable as an object of $\mathrm{DM}(S)$.
		\item $M$ is dualisable as an object of $\mathrm{DN}(S)$.
		\item The $\ell$-adic realisation of $M$ is a lisse sheaf.
	\end{enumerate}
	That is, the natural functors $$\mathrm{Rep}^A_\Q(\pi^1_{\acute{e}t}(S))\to\mathrm{DM}^{\mathrm{rig}}_c(S)\cap \mathrm{DM}_c^0(S)^\heartsuit \to\mathrm{DN}_c^{0,\mathrm{rig}}(S)^\heartsuit\to \mathrm{LS}(S_{\acute{e}t},\Q_\ell)\times_{\mathrm{Shv}_{\acute{e}t}(S,\Q_\ell)} \mathrm{DN}^0_c(S)^\heartsuit$$ 
	are equivalences, where $\mathrm{LS}(S_{\acute{e}t},\Q_\ell)$ consists of $\Q_\ell$-adic étale local systems on $S$.
\end{prop}
\begin{proof}
	Of course, $1.\Rightarrow 2. \Rightarrow 3. \Rightarrow 4.$ because of Poincaré duality for finite étale morphisms and because a symmetric monoidal functor preserves dualisable objects. 

	We prove that $4.\Rightarrow 1.$ by adapting Johann Haas' \cite[Lemma 6.12]{HaasThesis} to the simpler case of $0$-motives. Let $\eta$ be the generic point of $S$.
	We first claim that the functor $\mathrm{DN}^{0,\mathrm{rig}}_c(S)^\heartsuit \to \mathrm{DN}^{0,\mathrm{rig}}_c(\eta)^\heartsuit$ is fully faithful. 
	Indeed for any dense open subset $j:U\to S$, the combination of \Cref{omega0support} and of the localisation sequence 
	gives that for $M\in\mathrm{DM}^{0,\mathrm{rig}}(S)$, $M\to \omega^0(j_*j^*M)$ is an equivalence as $\omega^0(i^!M)= 0$. Therefore the unit of the adjunction $(j^*,\tau^{\leqslant 0}\omega^0j_*)$ on $\mathrm{DN}^0_c(S)^\heartsuit$ is an equivalence on $\mathrm{DN}^{0,\mathrm{rig}}(S)^\heartsuit$, 
	this means that $j^*$ is fully faithful on $\mathrm{DN}^{0,\mathrm{rig}}(S)^\heartsuit$. Passing to the colimit, $\eta^*$ is fully faithful.

	Therefore there is a commutative diagram 
	\begin{equation}
		\begin{tikzcd}
			\mathrm{Rep}^A_\Q(\pi^1_{\acute{e}t}(S)) \ar[r] \ar[d] & \mathrm{DN}^{0,\mathrm{rig}}_c(S)^\heartsuit \ar[d] \\
			\mathrm{Rep}^A_\Q(\pi^1_{\acute{e}t}(\eta)) \ar[r] & \mathrm{DN}^{0}_c(\eta)^\heartsuit
		\end{tikzcd}
	\end{equation}
	\label{goodRed}
	in which the lower horizontal functor is an equivalence of categories, the top functor is fully faithful by \Cref{thmA} and the right vertical functor is fully faithful by what we've just seen. 
	The left vertical functor is then also fully faithful.
    But then by \Cref{GoodRed} if $M\in\mathrm{DN}_c^{0,\mathrm{rig}}(S)^\heartsuit$, we can lift to to $S$ its pre-image in $\mathrm{Rep}^A_\Q(\pi_{\acute{e}t}^1(\eta))$ and the top functor is also essentially surjective.

	Therefore if $M\in\mathrm{DM}_c^0(S)^\heartsuit$ is such that its $\ell$-adic realisation is lisse (\emph{i.e.} dualisable), then its image in the category of perverse Nori motives is in $\mathrm{DN}^{0,\mathrm{rig}}_c(S)^\heartsuit$, 
	which is equivalent to $\mathrm{Rep}^A_\Q(\pi^1_{\acute{e}t}(S))$: our object $M$ is in $\mathrm{DM}^{0,\mathrm{sm}}_c(S)^\heartsuit$, and the proof is finished.
\end{proof}

\begin{cor}
	Let $S$ be a regular $k$-scheme of finite type.
	Let $M\in\mathrm{DM}^0_c(S)=\mathrm{DN}^0_c(S)$. Then the following are equivalent :
	\begin{enumerate}
		\item $M\in\mathrm{DM}^{0,\mathrm{sm}}(S)=\Dd^b(\mathrm{Rep}^A_\Q(\pi^1_{\acute{e}t}(S)))$.
		\item $M\in\mathrm{DM}(S)$ is dualisable.
		\item The $\ell$-adic realisation of $M$ is dualisable.
	\end{enumerate}
	Moreover, in that case, the dual of $M$ is also a $0$-motive. 
	That is, the inclusion $\mathrm{DM}_c^{0,\mathrm{sm}}(S)\to\mathrm{DM}_c^{0,\mathrm{rig}}(S)$ is an equivalence of 
	categories, with $\mathrm{DM}_c^{0,\mathrm{rig}}(S)$ the full subcategory of $M\in\mathrm{DM}_c^0(S)$ that are dualisable in $\mathrm{DM}(S)$.
\end{cor}
\begin{proof}
	$2.\Rightarrow 3.$ is obvious, and the proof of $1.\Rightarrow 2.$ follows by dévissage, 
	using that if $1.$, each $\HH^i(M)$ are dualisable, that $M$ is an iterated extension of 
	its $\HH^i(M)$ and that the full subcategory of dualisable 
	objects in a symmetric monoidal stable $\infty$-category is thick. The proof of $3.\Rightarrow 1.$ goes as 
	follow : the hypothesis on $M$ implies that $M$, seen as an object of $\Dd^b(\Mc(S))$, is dualisable, where we have denote by $\Mc(S)$ the ordinary heart of $\mathrm{DN}_c(S)$. Then because the t-structure restricts to dualisable objects
	 each $\HH^i(M)$ is dualisable hence 
	by the previous proposition, $\HH^i(M)\in\mathrm{DM}^{0,\mathrm{sm}}_c(S)$. As the latter category is thick and $M$ is an iterated 
	extension of its $\HH^i(M)$, we obtain that $M\in\mathrm{DM}^{0,\mathrm{sm}}_c(S)$. 
\end{proof}

In the case $S$ is only normal, following Haas we also have a result in that direction, but not as strong as the preceding corollary that would be false by \cite[Remark 2.1.2]{ruimyAbelianCategoriesArtin2023}. However this 
enables us to endow $\mathrm{DM}^{0,\mathrm{sm}}_c(S)$ with a t-structure.
\begin{cor}
	Let $S$ be a normal finite type $k$-scheme. The natural functors $\mathrm{DM}^{0,\mathrm{sm}}_c(S)\to\mathrm{DM}^{0,\mathrm{rig}}(S) \to\mathrm{DN}^{0,\mathrm{rig}}_c(S)$ are equivalences of $\infty$-categories. 
	If we transport the natural ordinary t-structure from $\mathrm{DN}^{0,\mathrm{rig}}_c(S)$ to $\mathrm{DM}^{0,\mathrm{sm}}_c(S)$, the inclusion $\mathrm{DM}^{0,\mathrm{sm}}_c(S)\to \mathrm{DM}^{0}_c(S)$ is t-exact, and the heart of $\mathrm{DM}^{0,\mathrm{sm}}_c(S)$ is 
	the category $\mathrm{Rep}^A_\Q(\pi^1_{\acute{e}t}(S))$.
\end{cor}
\begin{proof}
	We can assume $S$ to be connected.
	By \Cref{thmA} the two functors are fully faithful. Therefore it suffices to show that the composition functor is essentially surjective. By dévissage it suffices to show that any object $M\in\mathrm{DN}^{0,\mathrm{rig}}_c(S)^\heartsuit$ is in the image of $\mathrm{DM}_c^{0,\mathrm{sm}}(S)\to\mathrm{DN}_c^{0,\mathrm{rig}}(S)$. 
	For this, we are going to show that the natural functor 
	$\mathrm{Rep}^A_\Q(\pi^1_{\acute{e}t}(S))\to \mathrm{DN}^{0,\mathrm{rig}}_c(S)^\heartsuit$ is an equivalence of categories. Pick a dense regular open subset $U$ of $S$, with immersion $j:U\to S$. We have the following commutative diagram :
	\begin{equation}
		\begin{tikzcd}
			\mathrm{Rep}^A_\Q(\pi^1_{\acute{e}t}(S)) \ar[r] \ar[d] & \mathrm{DN}^{0,\mathrm{rig}}_c(S)^\heartsuit \ar[d] \\
			\mathrm{Rep}^A_\Q(\pi^1_{\acute{e}t}(U)) \ar[r] & \mathrm{DN}^{0,\mathrm{rig}}_c(U)^\heartsuit
		\end{tikzcd}.
	\end{equation}
	In this diagram, the lower horizontal functor is an equivalence of categories. The left vertical functor is fully faithful because both $U$ and $S$ are normal and $j$ is dominant. The top horizontal functor is fully faithful because the right vertical functor is faithful (it can be check on the $\ell$-adic realisation). Now, the top functor is also essentially surjective hence an equivalence thanks to \Cref{GoodRed}. 
	This shows  $\mathrm{Rep}^A_\Q(\pi^1_{\acute{e}t}(S)) \to \mathrm{DM}_c^{0,\mathrm{sm}}(S)$ is fully faithful and that
	 $\mathrm{Rep}^A_\Q(\pi^1_{\acute{e}t}(S)) \to \mathrm{DN}_c^{0,\mathrm{rig}}(S)^\heartsuit$ is an equivalence of categories. Therefore, $\mathrm{DM}_c^{0,\mathrm{sm}}(S)$ has a t-structure whose heart 
	 is $\mathrm{Rep}^A_\Q(\pi^1_{\acute{e}t}(S))$. Now, if $M\in\mathrm{DM}^{0,\mathrm{rig}}_c(S)$, then $\rho_\mathrm{N}(M)$ has all its cohomology sheaves in $\mathrm{Rep}^A_\Q(\pi^1_{\acute{e}t}(S))$, hence $M\in\mathrm{DM}_c^{0,\mathrm{sm}}(S)$.  This finishes the proof.

\end{proof}

We now apply this computation of $0$-motives to motivic Galois groups.

\begin{nota}
	Let $\sigma : \overline{k}\to\C$ be an embedding. Let $X$ be a finite type $k$-scheme. We denote by $\mcal^{\mathrm{rig}}(X)$ the subcategory of
	$\Mc(X)$ whose objects are the dualisable objects. This means that for $M\in \Mc(X)$, we have $M\in \mcal^{\mathrm{rig}}(X)$ if and only if $M$ is dualisable and its dual is in $\Mc(X)$ (that last condition is automatic).
\end{nota}
\begin{defi}
	\label{univCover}
	Let $X$ be a connected and normal finite type $k$-scheme with a geometric point $\overline{x}:\Spec(\overline{k})\to X$. We define
	$p:\widetilde{X}\to X$ to be the limit of all finite étale $f:Y\to X$ that factor $\overline{x}$. The limit exists as a scheme because the
	transitions morphisms have to be affine. There is a morphism $\pi_{\widetilde{X}} : \widetilde{X}\to\Spec(\overline{k})$ compatible with $p$, which has a section $\widetilde{x}$. Therefore
	there is also a map $q : \widetilde{X}\to X_{\overline{k}}$. By construction, the point $\overline{x}$ factors through $\widetilde{X}$ to give a geometric point $\tilde{x}\colon \Spec(\overline{k})\to\widetilde{X}$.
\end{defi}
\begin{lem}
  The scheme $\widetilde{X}$ is connected.
\end{lem}
\begin{proof}
  If $U\subset \widetilde{X}$ is a clopen subset such that $\overline{x}$ does not factor through $U$, then there exists $Y\to X$ finite étale factorising $\overline{x}$ and a clopen $V$ of $Y$ whose pullback to $\widetilde{X}$ is $U$. Let $Z$ be the open complement of $V$ in $Y$. Then $\overline{x}$ factors through $V$ and $Z\to X$ is finite étale, so that $\widetilde{X}=\lim_i Y_i$ with $Y_i\to Z$ finite étale factoring $\overline{x}$, and we see that $U=\emptyset$.
\end{proof}
\begin{rem}
	We defined $\widetilde{X}$ for a normal scheme $X$ because of the application we have in mind. Of course the definition makes sense for any qcqs scheme. Although the author did not check the details, it is probable that in the normal case the scheme $\widetilde{X}$ is the normalisation in $\Spec(L)$ of $X$, where $L$ is the maximal algebraic extension of $k(X)$ which is unramified over $X$.
\end{rem}
Base change gives a monoidal functor $p^* : \mathrm{DN}_c(X)\to\mathrm{DN}_c(\widetilde{X})$ which sends $\mcal^\mathrm{rig}(X)$ to $\mcal^\mathrm{rig}(\widetilde{X})$.
\begin{lem}
	Let $x:\Spec(\overline{k})\to X$ be a geometric point of a connected finite type $k$ scheme. Let $\omega_x : \mcal^\mathrm{rig}(X)\to\mathrm{Vect}_\Q$ be the composition of $x^*$ with the Betti realisation (through $\sigma$).
	Then $\mcal^\mathrm{rig}(X)$ is a Tannakian category with fiber functor $\omega_x$.
\end{lem}
\begin{proof}
	Every object of $\mcal(X)$ is dualsiable, the endomorphisms of the unit objects are one dimensional and the functor $\omega_x$ is monoidal.
\end{proof}
\begin{lem}
	Let $X$ and $\widetilde{X}$ as in \Cref{univCover}. The category $\mcal^\mathrm{rig}(\widetilde{X})$ together with the fiber functor 
	\[\omega_{\tilde{x}}\colon \mcal^\mathrm{rig}(\widetilde{X})\to \mathrm{Vect}_\Q\] is neutral Tannakian.
\end{lem}
\begin{proof}
	The category $\mcal^\mathrm{rig}(\widetilde{X})$ is the colimit of all $\mcal^\mathrm{rig}(Y)$ with $Y$ finite étale over $X$. Hence all of its objects are dualisable. As $\widetilde{X}$ is connected, the category $\mcal^\mathrm{rig}(\widetilde{X})$ is Tannakian. It is neutralised by $\tilde{x}$.
\end{proof}
\begin{nota}
	Let $X$ be a connected normal finite type $k$-scheme with a geometric point $x$. We denote by $\mathrm{G}_\mathrm{mot}(X,x,\sigma)$ the Tannakian group of $\mcal^\mathrm{rig}(X)$ and by $\mathrm{G}_\mathrm{mot}(\widetilde{X},\tilde{x},\sigma)$ the Tannakian group of $\mcal^\mathrm{rig}(\widetilde{X})$
\end{nota}

Let $s:X_{\overline{k}}\to X$ be the projection.

\begin{prop}
	\label{exactDiagMotGroup}
	Let $X$ be a normal connected finite type $k$-scheme. The commutative diagram
	\begin{equation}
		\begin{tikzcd}
			{\mathrm{DN}^{0,\mathrm{rig}}(k)^\heartsuit} & {\mathrm{DN}^{0,\mathrm{rig}}(X)^\heartsuit} & {\mathrm{DN}^{0,\mathrm{rig}}(X_{\overline{k}})^\heartsuit} \\
			{\mathrm{DN}^{0,\mathrm{rig}}(k)^\heartsuit} & {\mcal^\mathrm{rig}(X)} & {\mcal^\mathrm{rig}(X_{\overline{k}})} \\
			{\mathrm{DN}^{0,\mathrm{rig}}(\overline{k})^\heartsuit}=\mathrm{Vect}_\Q \ar[r,"\pi_{\widetilde{X}}^*"]& {\mcal^\mathrm{rig}(\widetilde{X})} & {\mcal^\mathrm{rig}(\widetilde{X})}
			\arrow[from=1-2, to=2-2]
			\arrow["{s^*}", from=1-3, to=2-3]
			\arrow["{p^*}", from=2-2, to=3-2]
			\arrow["{q^*}", from=2-3, to=3-3]
			\arrow["{s^*}", from=2-2, to=2-3]
			\arrow[from=3-2, to=3-3,"\mathrm{Id}"]
			\arrow["{\pi_X^*}", from=1-1, to=1-2]
			\arrow[from=1-2, to=1-3]
			\arrow[from=1-1, to=2-1,"\mathrm{Id}"]
			\arrow[from=2-1, to=2-2]
			\arrow[from=2-1, to=3-1]
			\arrow[from=3-1, to=3-2]
			\arrow["{\pi_X^*}", from=2-1, to=2-2]
		\end{tikzcd}
	\end{equation}
	of categories gives by taking Tannakian duals in the following commutative diagram
	\begin{equation}
		\begin{tikzcd}
			& 0 & 0 & {} \\
			0 & { \mathrm{G}_\mathrm{mot}(\widetilde{X},\widetilde{x},\sigma)} & { \mathrm{G}_\mathrm{mot}(\widetilde{X},\widetilde{x},\sigma)} & 0 \\
			0 & { \mathrm{G}_\mathrm{mot}(X_{\overline{k}},\overline{x},\sigma)} & { \mathrm{G}_\mathrm{mot}(X,x,\sigma)} & {\mathrm{Gal}(\overline{k}/k)} & 0 \\
			0 & {\pi_1^{\acute{e}t}(X_{\overline{k}},\overline{x})} & {\pi_1^{\acute{e}t}(X,x)} & {\mathrm{Gal}(\overline{k}/k)} & 0 \\
			& 0 & 0 & 0
			\arrow[from=2-2, to=2-3]
			\arrow[from=2-3, to=2-4]
			\arrow[from=2-1, to=2-2]
			\arrow[from=1-2, to=2-2]
			\arrow[from=1-3, to=2-3]
			\arrow[from=3-1, to=3-2]
			\arrow[from=2-2, to=3-2]
			\arrow[from=2-3, to=3-3]
			\arrow[from=3-2, to=3-3]
			\arrow[from=4-2, to=4-3]
			\arrow[from=4-1, to=4-2]
			\arrow[from=3-2, to=4-2]
			\arrow[from=3-3, to=4-3]
			\arrow[from=4-3, to=4-4]
			\arrow[from=4-4, to=4-5]
			\arrow[from=3-3, to=3-4]
			\arrow[from=3-4, to=3-5]
			\arrow[from=2-4, to=3-4]
			\arrow[from=3-4, to=4-4]
			\arrow[from=4-4, to=5-4]
			\arrow[from=4-2, to=5-2]
			\arrow[from=4-3, to=5-3]
		\end{tikzcd}
	\end{equation}
	in which rows and columns are exact.
\end{prop}
\begin{proof}
	Taking Tannakian duals indeed give the commutative diagram of pro-algebraic $\Q$-groups by \Cref{pi1etale}.

	The facts that the second maps in each column and each row is faithfully flat is equivalent to the fact that the maps between
	Tannakian categories are fully faithful and the image is closed under sub-quotients. Therefore it suffices to check that 
	$\mathrm{Rep}^A_\Q(\mathrm{Gal}(\overline{k}/k))\to\mathrm{Rep}_\Q^A(\pi_{\acute{e}t}^1(X))$ and $\mathrm{Rep}_\Q^A(\pi_{\acute{e}t}^1(S))\to \mcal^{\mathrm{rig}}(X)$ are fully faithful
	with image closed under sub-quotients. For the first map, this is \cite{SGA1}. The second is fully faithful, we have to check that it is closed under 
	sub-quotients. Working dually it suffices to check that the category $\mathrm{Rep}_\Q^A(\pi_{\acute{e}t}^1(X))$ is stable under subobjects in $\mcal^{\mathrm{rig}}(X)$. 
	Let $M\to N$ be a monomorphism in $\mcal^{\mathrm{rig}}(X)$ with $N$ a $0$-motive. By \Cref{stabnmot} to show that $M$ is a $0$-motive it suffices to check that for all $x\in X$ the 
	motive $M_x$ is a $0$-motive. But by \cite[Theorem 9.1.16]{MR3618276} the category of $0$-motives is closed under sub-quotients in the categories of motives over a field.

	We prove the middle exactitude and injectivity at the same time. As the proofs are the same, we only deal with the sequence
	\begin{equation}
		\mathrm{DN}^{0,\mathrm{sm}\heartsuit}(X)\to\mcal^\mathrm{rig}(X)\to \mcal^\mathrm{rig}(\widetilde{X})
	.\end{equation}
	As in the proof of \cite[Theorem 9.1.16 and Erratum p. 234]{MR3618276}, we first prove that any motive $M\in \mcal^\mathrm{rig}(\widetilde{X})$ is a direct
	factor of the image of a motive $N \in \mcal^\mathrm{rig}(X)$. Indeed, if $M\in \mcal^\mathrm{rig}(\widetilde{X})$, there is a finite étale morphism $f:Y\to X$ (that we can assume to be Galois) such that
	$M$ is defined over $Y$ that is $M \in \mcal^\mathrm{rig}(Y)$. But as the composition $M\to f^*f_*M = f^!f_!M\to M$ is the multiplication by $\deg f$, $M$ is a direct factor of $f^*N$ with $N = f_*N \in \mcal^\mathrm{rig}(X)$.

	Let $\mathcal{U}$ be the Tannakian full subcategory of $\mcal^\mathrm{rig}(\widetilde{X})$ generated by the unit object $\Q$. We say that elements of $\mathcal{U}$ are \emph{trivial}. To finish the proof
	we just have to prove that $M\in\mcal^{\mathrm{rig}}(X)$ has trivial image in $\mcal^\mathrm{rig}(\widetilde{X})$ if and only if it is in $\mathrm{DN}^{0,\mathrm{sm}\heartsuit}(X)$.
	Of course, as any representation of $\pi_1^{\acute{e}t}(X)$ is trivialised after a base change to a finite étale cover of $X$, the image of Artin motives in $\mcal^\mathrm{rig}(\widetilde{X})$ is trivial.
	Conversely, let $M\in\mcal^\mathrm{rig}(X)$ with trivial image in $\mcal^\mathrm{rig}(\widetilde{X})$. This means that there is a finite étale morphism $f:Y\to X$ such that $f^*M$ is trivial.
	We can assume that $f$ is Galois. We have that $f_*f^*M$ is in the category generated by $f_*\Q_Y$, which is inside $\mathrm{DN}^{0,\mathrm{sm}\heartsuit}(X)$. As $M$ is a direct factor of $f_*f^*M$, we get that $M$ is an Artin motive.

\end{proof}
\subsection*{Acknowledgements} I would like to thank my advisor Sophie Morel for many discussions and remarks. I also thank Raphaël Ruimy for comments on drafts of this note and Frédéric Déglise for helpful discussions. 
The author was funded by the ÉNS de Lyon and by the ANR-21-CE40-0015 project HQDIAG.
\bibliographystyle{alpha}
\bibliography{BibSCNet}

\begin{thebibliography}{EHIK21}

\bibitem[ABV09]{MR2494373}
Joseph Ayoub and Luca Barbieri-Viale.
\newblock 1-motivic sheaves and the {A}lbanese functor.
\newblock {\em J. Pure Appl. Algebra}, 213(5):809--839, 2009.

\bibitem[ABV15]{MR3302623}
Joseph Ayoub and Luca Barbieri-Viale.
\newblock Nori 1-motives.
\newblock {\em Math. Ann.}, 361(1-2):367--402, 2015.

\bibitem[Aut18]{stacks-project}
The Stacks~Project Authors.
\newblock Stacks project, 2018.
\newblock \url{https://stacks.math.columbia.edu/}.

\bibitem[Ayo14]{MR3205601}
Joseph Ayoub.
\newblock La r\'{e}alisation \'{e}tale et les op\'{e}rations de {G}rothendieck.
\newblock {\em Ann. Sci. \'{E}c. Norm. Sup\'{e}r. (4)}, 47(1):1--145, 2014.

\bibitem[CD16]{MR3477640}
Denis-Charles Cisinski and Fr\'{e}d\'{e}ric D\'{e}glise.
\newblock \'{E}tale motives.
\newblock {\em Compos. Math.}, 152(3):556--666, 2016.

\bibitem[CDN23]{MR4631432}
M.~Cavicchi, F.~D\'eglise, and J.~Nagel.
\newblock Motivic decompositions of families with {T}ate fibers: smooth and
  singular cases.
\newblock {\em Int. Math. Res. Not. IMRN}, (16):14239--14289, 2023.

\bibitem[EHIK21]{MR4319065}
Elden Elmanto, Marc Hoyois, Ryomei Iwasa, and Shane Kelly.
\newblock Milnor excision for motivic spectra.
\newblock {\em J. Reine Angew. Math.}, 779:223--235, 2021.

\bibitem[Gro61]{MR0217085}
A.~Grothendieck.
\newblock \'{E}l\'{e}ments de g\'{e}om\'{e}trie alg\'{e}brique. {III}.
  \'{E}tude cohomologique des faisceaux coh\'{e}rents. {I}.
\newblock {\em Inst. Hautes \'{E}tudes Sci. Publ. Math.}, (11):167, 1961.

\bibitem[Haa19]{HaasThesis}
Johann Haas.
\newblock {\em Lisse $1$-motives}.
\newblock PhD thesis, The University of Regensburg, 2019.
\newblock available at
  \url{https://epub.uni-regensburg.de/43953/1/Thesis%20for%20bib.pdf}.

\bibitem[HMS17]{MR3618276}
Annette Huber and Stefan M\"{u}ller-Stach.
\newblock {\em Periods and {N}ori motives}, volume~65 of {\em Ergebnisse der
  Mathematik und ihrer Grenzgebiete. 3. Folge. A Series of Modern Surveys in
  Mathematics [Results in Mathematics and Related Areas. 3rd Series. A Series
  of Modern Surveys in Mathematics]}.
\newblock Springer, Cham, 2017.
\newblock With contributions by Benjamin Friedrich and Jonas von Wangenheim.

\bibitem[IM22]{ivorraFourOperationsPerverse2022}
Florian Ivorra and Sophie Morel.
\newblock The four operations on perverse motives.
\newblock Preprint, available at
  \url{http://perso.ens-lyon.fr/sophie.morel/PerverseMotives.pdf}, 2022.

\bibitem[Mor08]{MR2350050}
Sophie Morel.
\newblock Complexes pond\'{e}r\'{e}s sur les compactifications de
  {B}aily-{B}orel: le cas des vari\'{e}t\'{e}s de {S}iegel.
\newblock {\em J. Amer. Math. Soc.}, 21(1):23--61, 2008.

\bibitem[NV15]{MR3293216}
Arvind~N. Nair and Vaibhav Vaish.
\newblock Weightless cohomology of algebraic varieties.
\newblock {\em J. Algebra}, 424:147--189, 2015.

\bibitem[Org04]{MR2102056}
Fabrice Orgogozo.
\newblock Isomotifs de dimension inf\'{e}rieure ou \'{e}gale \`a un.
\newblock {\em Manuscripta Math.}, 115(3):339--360, 2004.

\bibitem[PL19a]{MR4033829}
Simon Pepin~Lehalleur.
\newblock Constructible 1-motives and exactness of realisation functors.
\newblock {\em Doc. Math.}, 24:1721--1737, 2019.

\bibitem[PL19b]{MR3920833}
Simon Pepin~Lehalleur.
\newblock Triangulated categories of relative 1-motives.
\newblock {\em Adv. Math.}, 347:473--596, 2019.

\bibitem[Rob14]{robaloThese}
Marco Robalo.
\newblock {\em Théorie homotopique motivique des espaces non-commutatifs.}
\newblock PhD thesis, University of Montpellier, 2014.

\bibitem[RT24]{integralNori}
Raphaël Ruimy and Swann Tubach.
\newblock Nori motives (and mixed hodge modules) with integral coefficients.
\newblock Preprint, available at \url{https://swann.tubach.fr/en/research/},
  2024.

\bibitem[Rui23]{ruimyAbelianCategoriesArtin2023}
Raphaël Ruimy.
\newblock {\em Abelian Categories of {{Artin}} étale {{Motives}} with Integral
  Coefficients}.
\newblock PhD thesis, ENS de Lyon, 20223.
\newblock available at \url{https://raphaelruimy2.wixsite.com/rruimy/research}.

\bibitem[Rui23]{ruimyArtinPerverseSheaves2023}
Raphaël Ruimy.
\newblock Artin perverse sheaves, 2023.
\newblock Preprint, available at \url{http://arxiv.org/abs/2205.07796}.

\bibitem[Sai90]{MR1047415}
Morihiko Saito.
\newblock Mixed {H}odge modules.
\newblock {\em Publ. Res. Inst. Math. Sci.}, 26(2):221--333, 1990.

\bibitem[SGA71]{SGA1}
{\em Rev\^{e}tements \'{e}tales et groupe fondamental.}
\newblock Springer-Verlag, Berlin-New York,,, 1971.
\newblock S\'{e}minaire de G\'{e}om\'{e}trie Alg\'{e}brique du Bois Marie
  1960--1961 (SGA 1)., Dirig\'{e} par Alexandre Grothendieck. Augment\'{e} de
  deux expos\'{e}s de M. Raynaud.

\bibitem[Ter24]{terenziTensorStructurePerverse2024}
Luca Terenzi.
\newblock Tensor structure on perverse {Nori} motives, 2024.
\newblock Available at \url{http://arxiv.org/abs/2401.13547}.

\bibitem[Tub23]{SwannRealisation}
Swann Tubach.
\newblock On the {{Nori}} and {{Hodge}} realisations of {{Voevodsky}} étale
  motives.
\newblock Preprint, available at \url{https://swann.tubach.fr/en/research/},
  2023.

\bibitem[Vai20]{MR4109490}
Vaibhav Vaish.
\newblock Punctual gluing of {$t$}-structures and weight structures.
\newblock {\em Manuscripta Math.}, 162(3-4):341--366, 2020.

\bibitem[VSF00]{MR1764197}
Vladimir Voevodsky, Andrei Suslin, and Eric~M. Friedlander.
\newblock {\em Cycles, transfers, and motivic homology theories}, volume 143 of
  {\em Annals of Mathematics Studies}.
\newblock Princeton University Press, Princeton, NJ, 2000.

\end{thebibliography}

Swann Tubach \url{swann.tubach@ens-lyon.fr}

E.N.S Lyon, UMPA, 46 Allée d'Italie, 

69364 Lyon Cedex 07, France
\end{document}